\documentclass[12pt]{amsart}
\usepackage{amssymb,amsmath}
\usepackage[colorlinks]{hyperref}
\usepackage[all]{xy}

\newcommand{\ra}{\rightarrow}

\newcommand{\bC}{\mathbb C}
\newcommand{\bF}{\mathbb F}
\newcommand{\bP}{\mathbb P}
\newcommand{\bQ}{\mathbb Q}
\newcommand{\bZ}{\mathbb Z}

\newcommand{\cC}{\mathcal C}
\newcommand{\cDP}{\mathcal{DP}}
\newcommand{\cH}{\mathcal H}
\newcommand{\cI}{\mathcal I}
\newcommand{\cN}{\mathcal N}
\newcommand{\cS}{\mathcal S}
\newcommand{\cO}{\mathcal O}

\newcommand{\fD}{\mathfrak D}
\newcommand{\fS}{\mathfrak S}

\DeclareMathOperator{\Proj}{Proj}
\DeclareMathOperator{\Hilb}{Hilb}
\DeclareMathOperator{\HH}{H}
\DeclareMathOperator{\im}{im}
\DeclareMathOperator{\Sym}{Sym}

\newcommand{\Bl}{\operatorname{Bl}}
\newcommand{\Br}{\operatorname{Br}}
\newcommand{\BS}{\operatorname{BS}}
\newcommand{\cor}{\operatorname{cor}}
\newcommand{\Gal}{\operatorname{Gal}}
\newcommand{\Gr}{\operatorname{Gr}}
\newcommand{\Hom}{\operatorname{Hom}}
\newcommand{\NS}{\operatorname{NS}}

\newcommand{\PGL}{\operatorname{PGL}}

\theoremstyle{plain}
\newtheorem{prop}{Proposition}
\newtheorem{theo}[prop]{Theorem}
\newtheorem{coro}[prop]{Corollary}
\theoremstyle{remark}
\newtheorem{rema}[prop]{Remark}
\newtheorem{remas}[prop]{Remarks}
\newtheorem*{lemm}{Lemma}
\newtheorem*{note}{Note}
\theoremstyle{definition}
\newtheorem{defi}[prop]{Definition}

\author[Addington]{Nicolas Addington}
\address{Department of Mathematics \\
University of Oregon \\
Eugene, OR 97403-1222 \\
USA}
\email{adding@uoregon.edu}

\author[Hassett]{Brendan Hassett}
\address{Department of Mathematics \\
Brown University \\
Box 1917 \\
151 Thayer Street \\
Providence, RI 02912 \\
USA}
\email{bhassett@math.brown.edu}

\author[Tschinkel]{Yuri Tschinkel}
\address{Courant Institute\\
New York University \\
New York, NY 10012 \\
USA }
\address{Simons Foundation\\
160 Fifth Avenue\\
New York, NY 10010\\
USA}
\email{tschinkel@cims.nyu.edu}

\author[V\'arilly-Alvarado]{Anthony V\'arilly-Alvarado}
\address{Department of Mathematics \\
Rice University MS 136 \\
Houston, TX 77251-1892 \\
USA}
\email{av15@rice.edu}

\title[Cubic fourfolds]{Cubic fourfolds fibered in \\ sextic del Pezzo surfaces}

\begin{document}

\begin{abstract}
We exhibit new examples of rational cubic fourfolds, parametrized by a countably infinite union of codimension-two
subvarieties in the moduli space.  Our examples
are fibered in sextic del Pezzo surfaces over the projective plane; they are rational whenever the fibration has a rational section.
\end{abstract}

\maketitle

The rationality problem for complex cubic fourfolds has been studied by
many authors; see \cite[\S 1]{Has16} for background and references
to the extensive literature on this subject. The moduli space $\cC$
of smooth cubic fourfolds has dimension 20. Since the 1990's, 
the only cubic fourfolds {\em known} to be rational are:
\begin{itemize}
\item{Pfaffian cubic fourfolds and their limits~\cite{BRS}.  These form a divisor
$\cC_{14} \subset \cC$.}
\item{Cubic fourfolds containing a plane $P\subset X$ such that the
induced quadric surface fibration $\Bl_P(X)\ra \bP^2$
has an odd-degree multisection~\cite{Has99}.  These form a countably infinite union of codimension-two loci
$\bigcup \cC_{K} \subset \cC$, dense in the divisor $\cC_8$ parametrizing
cubic fourfolds containing a plane.}
\end{itemize}
Our main result is:
\begin{theo} \label{theo:main}
Let $\cC_{18} \subset \cC$ denote the divisor of cubic fourfolds of discriminant
$18$. There is a Zariski open subset $U \subset \cC_{18}$ and
a countably infinite union of codimension-two loci
$\bigcup \cC_{K} \subset \cC$, dense in $\cC_{18}$, such that
$\bigcup \cC_K \cap U$ parametrizes rational cubic fourfolds.
\end{theo}
We show that the generic element $X$ of $\cC_{18}$ is birational to a fibration in sextic del Pezzo surfaces over $\mathbb P^2$. Its generic fiber and thus $X$ are rational if the fibration admits a section. In fact, a multisection of degree prime to three suffices to establish rationality of $X$. This condition can be expressed in Hodge-theoretic terms, and is satisfied along a countably infinite union of divisors in $\cC_{18}$.

In Section~\ref{sect:sexticcubic} we show that a generic cubic fourfold in $\cC_{18}$ contains an elliptic ruled surface. In Section~\ref{sect:constructing} we construct fibrations in sextic del Pezzo surfaces from elliptic ruled surfaces.  After reviewing the rationality of sextic del Pezzo surfaces in Section~\ref{sect:basics}, we prove Theorem~\ref{theo:main} in Section~\ref{sect:rationality}.  In Section~\ref{sect:fibrations} we analyze the degenerate fibers of fourfolds fibered in sextic del Pezzo surfaces.  Our approach is grounded in assumptions on the behavior of `generic' cases;
Section~\ref{sect:example} validates these in a concrete
example.

\bigskip

\noindent {\bf Acknowledgments:} We are grateful to Asher Auel and Alexander Perry for stimulating discussions on related topics, and to the referees who suggested many improvements. The second author is partially supported by NSF grant 1551514; the fourth author is supported by NSF grant 1352291.

\section{Sextic elliptic ruled surfaces and cubic fourfolds}
\label{sect:sexticcubic}

\begin{theo} \label{theo:sexticcubic}
A generic cubic $X \in \cC_{18}$ contains a surface $T$ of degree 6, with the following properties:
\begin{enumerate}
\item $T$ is an elliptic ruled surface, that is, it admits a $\bP^1$-fibration over an elliptic curve $E$;
\item the linear system of quadrics in $\bP^5$ containing $T$ is 2-dimensional;
\item the base locus of the linear system is a complete intersection $\Pi_1 \cup T \cup \Pi_2$, where $\Pi_1$ and $\Pi_2$ are disjoint planes; and
\item the curves $E_i := \Pi_i \cap T$ give sections of the $\bP^1$-fibration $T \to E$.
\end{enumerate}
\end{theo}

\noindent We do not claim that the set of cubics $X \in \cC_{18}$ for which these conclusions hold is open, only that it contains a non-empty open set; see \cite{BRS} for comparable subtleties in $\cC_{14}$.

Before proving this we prove two preliminary results.  We take the planes $\Pi_1, \Pi_2$ as our starting point.

\begin{prop} \label{prop:two_planes}
Let $\Pi_1, \Pi_2 \subset \bP^5$ be two disjoint planes.  Then for a generic choice of three quadrics $Q_1, Q_2, Q_3$ containing $\Pi_1 \cup \Pi_2$, we have
\[ Q_1 \cap Q_2 \cap Q_3 = \Pi_1 \cup T \cup \Pi_2, \]
where $T$ is an elliptic ruled surface of degree 6.  Moreover, the curves $E_i := T \cap \Pi_i$ are sections of the $\bP^1$-fibration $T \to E$.
\end{prop}
\begin{proof}
Observe that a point in $\bP^5 \setminus (\Pi_1 \cup \Pi_2)$ is contained in a unique line that meets $\Pi_1$ and $\Pi_2$, giving a rational map $\bP^5 \dashrightarrow \Pi_1 \times \Pi_2$.  The graph of this rational map is the blow-up $\tilde\bP^5$ of $\bP^5$ along $\Pi_1$ and $\Pi_2$.  Label the projections $p$ and $q$ as
\[ \xymatrix{
\tilde\bP^5 \ar[d]_-q \ar[r]^-p & \Pi_1 \times \Pi_2 \\
{\phantom{,}}\bP^5,
} \]
and observe that $p$ is a $\bP^1$-bundle, the projectivization of $\cO(-1,0) \oplus \cO(0,-1)$ over $\bP^1 \times \bP^1$.

Let $F_1, F_2 \subset \tilde\bP^5$ be the exceptional divisors, let $H$ be the hyperplane class on $\bP^5$, and let $H_1, H_2$ denote the pullbacks of the hyperplane classes on $\Pi_1 \times \Pi_2$.  We find that $p^* H_i = q^* H - F_i$, so $p^*(H_1 + H_2) = q^*(2H) - F_1 - F_2$; that is, divisors of type $(1,1)$ on $\Pi_1 \times \Pi_2$ correspond to quadrics in $\bP^5$ containing $\Pi_1$ and $\Pi_2$.

By Bertini's theorem and the adjunction formula, a generic intersection of three divisors of type $(1,1)$ is a smooth genus-one curve $E \subset \Pi_1 \times \Pi_2$.  The projection onto either factor maps $E$ isomorphically onto a smooth cubic curve $E_i \subset \Pi_i$, as follows.  A generic intersection of two $(1,1)$-divisors is a sextic del Pezzo surface in $\Pi_1 \times \Pi_2$, and the projection onto either plane contracts three lines.  The third $(1,1)$-divisor meets these lines in at most one point each (since $E$ contains no lines), so the projections $E \to E_i$ are isomorphisms.

Because $p$ is a $\bP^1$-bundle, $\tilde T := p^{-1}(E)$ is an elliptic ruled surface.  The exceptional divisors $F_1, F_2 \subset \tilde\bP^5$ give sections of $p$, hence sections of $\tilde T \to E$.  Let $T = q(\tilde T) \subset \bP^5$.  Because the projections $E \to E_i \subset \Pi_i$ are isomorphisms, it follows that $\tilde T \to T$ is an isomorphism and $T \cap \Pi_i = E_i$.

We have $\deg T = 6$ because $T$ is residual to a pair of planes in a complete
intersection of three quadrics.
\end{proof}

\begin{prop} \label{prop:homma_stuff}
Let $T \subset \bP^5$ be an elliptic ruled surface as in Proposition~\ref{prop:two_planes}.  Then the homogeneous ideal of $T$ is generated by quadrics and cubics.  Moreover,
\begin{align*}
h^0(\cI_T(1)) &= 0, &
h^0(\cI_T(2)) &= 3, &
h^0(\cI_T(3)) &= 20.
\end{align*}
\end{prop}
\begin{proof}
The first statement follows from \cite[Thm.~3.3]{Homma}.  The second can be found in \cite[\S2]{Homma}, or we can prove it as follows.  From the inclusion $T \subset (\Pi_1 \cup T \cup \Pi_2) \subset \bP^5$ we get an exact sequence
\[ 0 \to \cI_{(\Pi_1 \cup T \cup \Pi_2)/\bP^5} \to \cI_{T/\bP^5} \to \cI_{T/(\Pi_1 \cup T \cup \Pi_2)} \to 0. \]
Since $\Pi_1 \cup T \cup \Pi_2$ is a complete intersection of three quadrics, we can compute cohomology of the first term using a Koszul complex.  The third term is isomorphic to
\[ \cI_{E_1/\Pi_1} \oplus \cI_{E_2/\Pi_2} = \cO_{\Pi_1}(-3) \oplus \cO_{\Pi_2}(-3). \]
Then the calculation is straightforward.
\end{proof}

\begin{proof}[Proof of Theorem~\ref{theo:sexticcubic}] We retain
the notation of Proposition~\ref{prop:two_planes}.
We claim there exists a smooth, irreducible, open subset in
the Hilbert scheme
$$ \cH_T \subset \Hilb_{\bP^5}^{3n^2+3n}$$
of dimension 36,
parametrizing the elliptic ruled surfaces $T$ we described there.  (Here $3n^2+3n$ is the Hilbert polynomial of $T \subset \bP^5$.)
The pushforward of $\cI_{\Pi_1\cup \Pi_2}(2)$  -- the quadrics vanishing
on the planes $\Pi_1,\Pi_2 \subset \bP^5$ --
is a rank 9 vector bundle over a dense open subset of
$\Sym^2(\Gr(3,6))$. The relative Grassmannian of rank 3 vector subbundles of this rank 9 vector bundle gives a $\Gr(3,9)$-bundle over $\Sym^2(\Gr(3,6))$, 
and $\cH_T$ is isomorphic to an open subset of this $\Gr(3,9)$-bundle. To see this, observe that each $T$ arises from a \emph{unique} pair of planes $\Pi_1, \Pi_2$: by intersecting the three quadrics containing $T$ we recover $\Pi_1$ and $\Pi_2$. We have $\dim(\cH_T)=2\cdot 9 + 18=36$.

Next, letting $V$ denote the open subset of $\bP^{55}$ that parametrizes smooth cubic fourfolds, we claim that the Hilbert scheme of pairs
\[
H := \{(T,X) \mid T\subset X\} \subset \cH_T  \times V
\]
is smooth and irreducible of dimension 55.  To see this, note that for the projection $\pi_1\colon H \to \cH_T$, the fiber $\pi_1^{-1}([T])$ is $\bP(\HH^0(\cI_{T}(3)))$. By Proposition~\ref{prop:two_planes} and the last part of Proposition~\ref{prop:homma_stuff}, it follows that $H$ is an open subset of a $\bP^{19}$-bundle over $\cH_T$, and is thus 55-dimensional. It is non-empty -- we write down an explicit smooth cubic fourfold containing one of our elliptic ruled surfaces in Section~\ref{sect:example}.

We show that, modding out by the action of $\PGL(6)$, the second projection $\pi_2\colon H \to V$ induces a dominant map $H /\!/ \PGL(6) \to \cC_{18}$ from a $20$-dimensional quasi-projective variety to $\cC_{18}$.  Together with Propositions~\ref{prop:two_planes} and~\ref{prop:homma_stuff}, this will establish all the statements in Theorem~\ref{theo:sexticcubic}.

The image of $\pi_2$ is 54-dimensional.  To see this, first note that $X \in \im\pi_2$ is a special cubic fourfold, and hence lies in the preimage of some $\cC_d$ under the quotient map $V \to V/\!/\PGL(6)$, which is 54-dimensional. On the other hand, since $\pi_2^{-1}(X) = \Hilb_X^{3n^2+3n}$, which has tangent space $\HH^0(\cN_{T/X})$ at $T$, we have
\[
h^0(\cN_{T/X}) \geq \dim \pi_2^{-1}(X) \geq \dim H - \dim \im \pi_2 \geq 1.
\]
By semicontinuity, if $h^0(\cN_{T/X}) = 1$ for one point $[T]$, then this dimension is generically $1$, and hence $\dim \pi_2^{-1}(X) = 1$ for a general $X$ in the image of $\pi_2$, showing that $\dim \im \pi_2 = 54$. We show this is the case with an explicit example using Macaulay2 in Section~\ref{sect:example}.

Finally, we claim that for $X \in \im \pi_2$, the image of $X$ under the quotient map $V \to V/\!/\PGL(6)$ lands in $\cC_{18}$. The intersection form on a smooth cubic fourfold $X$ containing $T$ is as follows:
\[ \begin{array}{c|cc}
	& h^2 & T  \\
\hline
h^2 & 3 & 6 \\
T   & 6 & 18 
\end{array} \]
Only $T^2 = 18$ needs to be justified.  By \cite[\S4.1]{Has00}, we have
\[ T^2 = 6h_T^2 + 3h_T.K_T + K_T^2 - \chi_T, \]
where $\chi_T$ is the topological Euler characteristic and $h_T$ is the hyperplane class restricted to $T$.  Because $T$ is an elliptic ruled surface, we have $K_T^2 = 0 = \chi_T$.
Moreover we have $K_T = -E_1 - E_2$: to see this, note that $K_T + E_1 + E_2$ has degree zero on the $\bP^1$ fibers of $T \to E$, hence is pulled back from $E$, and its restriction to a section ($E_1$ or $E_2$) is zero by adjunction, so $K_T + E_1 + E_2 = 0$.  Thus we have $h_T.K_T = -6$.

All together, we obtain the dominant map 
\[
H /\!/ \PGL(6) \to \cC_{18},
\]
which concludes the proof of Theorem~\ref{theo:sexticcubic}.
\end{proof}

\begin{rema}
The ruled surface $T \subset X$ determines an elliptic curve $E$ in the variety of lines $F_1(X)$.  By \cite[Cor.~5.1]{Ran}, such a curve moves in a family of dimension at least 1. Thus one expects that $T$ deforms inside $X$
in a one-parameter family.
\end{rema}

\section{Constructing the del Pezzo fibration}
\label{sect:constructing}

\begin{theo} \label{theo:fibration}
Let $X$ be a cubic fourfold containing an elliptic ruled surface $T$ as in Theorem~\ref{theo:sexticcubic}, and let
\[ \pi\colon X' := \Bl_T(X) \to \bP^2 \]
be the map induced by the linear system of quadrics containing $T$.  For generic $X$, the generic fiber of $\pi$ is a del Pezzo surface of degree 6, and the preimages of the curves $E_1, E_2 \subset T$ induce trisections of $\pi$.
\end{theo}
\begin{proof}
We will show that for a generic choice of two quadrics $Q_1, Q_2 \supset T$ and a cubic $X \supset T$, we have $Q_1 \cap Q_2 \cap X = T \cup S$, where $S$ is a del Pezzo surface of degree 6.

We have seen that any quadric containing $T$ also contains the two planes $\Pi_1, \Pi_2$ appearing in Theorem~\ref{theo:sexticcubic}.  But by Proposition~\ref{prop:homma_stuff} we know that $T$ is cut out by cubics, so a generic cubic containing $T$ does not contain $\Pi_1$ or $\Pi_2$.  Thus by Bertini's theorem, we can say that $Q_1 \cap Q_2 \cap X = T \cup S$, where $S$ is irreducible and smooth away from $T$.  First we argue that $S$ is in fact smooth everywhere, and that $S$ and $T$ meet transversely in a smooth curve $D$ (see Remark~\ref{rema:D}).

Let $V = Q_1 \cap Q_2$.  We claim that $V$ has ordinary double points
$$s_{11},s_{12},s_{13} \in \Pi_1 \cap T, \quad
s_{21},s_{22},s_{23} \in \Pi_2 \cap T.$$
(Compare \cite[Thm.~2.1]{Kap}.)  
It is helpful to keep in mind the characteristic example
\begin{align*}
\Pi_1=\{x_1=x_3=x_5=0\},\ & \Pi_2=\{x_0=x_2=x_4=0\} \subset & \\
V&=\{x_0x_1-x_2x_3=x_2x_3-x_4x_5=0\},
\end{align*}
which has ordinary singularities on the coordinate axes.
To see the claim, let
\[ \xymatrix{
\tilde\bP^5 \ar[d]_-q \ar[r]^-p & \Pi_1 \times \Pi_2 \\
{\phantom{,}}\bP^5,
} \]
be as in the proof of Proposition~\ref{prop:two_planes}, and let $W \subset \Pi_1 \times \Pi_2$ be the intersection of two $(1,1)$-divisors corresponding to $Q_1, Q_2$.  By adjunction, $W$ is a smooth sextic del Pezzo surface; the projections $W \to \Pi_i$ each contract three lines. Now the $\bP^1$-bundle
$$p_W\colon \tilde{V}:=p^{-1}(W) \ra W$$
resolves the singularities of $V$ via $q$; the proper
transforms of the $\Pi_i$ are sections of $p_W$. Thus $V$ has six
singularities, resolved by blowing up the $\Pi_i$. Our example
shows these are no worse than ordinary singularities.

Now consider the blow-up
\[ r\colon \Bl_T(\bP^5) \to \bP^5, \]
let $F \subset \Bl_T(\bP^5)$ be the exceptional divisor, and let $V' \subset \Bl_T(\bP^5)$ be the proper transform of $V$.  This is a small resolution of $V$; indeed, it is the flop of $\tilde V $ along the six exceptional lines.  Let $T' = V' \cap F$, which is the blow-up of $T$ at the six points $s_{ij}$.

Let $X'$ be a generic divisor in the linear system $|r^*(3H) - F|$. 
Proposition~\ref{prop:homma_stuff} guarantees that $T$ is cut out
by the cubics in its homogeneous ideal. Blowing up $T$ resolves 
the indeterminacy of this linear system, so it is basepoint-free.
Thus $X'$ is smooth, and moreover the surface $S' := X' \cap V'$ and the curve $D' := X' \cap T'$ are smooth.  We claim that $r$ maps $S'$ isomorphically onto $S = r(S')$, and $D'$ isomorphically onto $D = r(D')$.  For this it is enough to observe that $X'$ meets the exceptional lines of $V' \to V$ in one point each; otherwise one of the lines would be contained in the base locus of $|r^*(3H) - F|$, which is impossible.  (As $Q_1$ and $Q_2$ vary, these six points give rise to the two trisections of $\pi$.)

Finally, we argue that $S$ is a del Pezzo surface of degree 6.  By adjunction we have $K_{S'} = r^* H - F$, so $K_S = H - D$.  But $D = S \cap T = S \cap Q_3$, so $K_S = H - 2H = -H$, so $S$ is a del Pezzo surface.  And $S \cup T = Q_1 \cap Q_2 \cap X$ has degree 12, and $T$ has degree 6, so $S$ has degree 6.
\end{proof}

\begin{rema}\label{rema:D}
The curve $D = T \cap S$ appearing in the previous proof is smooth of degree $12$ and genus $7$.  To see this, note again that on $S$ we have $D \sim 2H$, so $D.H = 12$ and $\deg K_D = K_S.D + D^2 = 12$.
\end{rema}

\section{Rationality of sextic del Pezzo surfaces}
\label{sect:basics}

To determine when the total space of our fibrations in sextic del Pezzo surfaces
$\pi\colon X' \to \bP^2$ are rational, we review 
rationality properties of sextic del Pezzo surfaces in general.

Let $S$ be a del Pezzo surface of degree six over a perfect field $F$.
(In our application $F$ is the function field of $\bP^2$.)
Over an algebraic closure $\bar{F}$, the surface $\bar{S}=S_{\bar F}$ is
isomorphic to $\bP^2_{\bar F}$ blown up at three non-collinear points.  Famously, $\bar{S}$ contains a hexagon of lines ($(-1)$-curves), consisting of the exceptional divisors and the proper transforms of the lines joining pairs of the points.  The standard Cremona transformation along the three points gives a second realization as a blow-up of $\bP^2_{\bar F}$ and exchanges the two kinds of lines; we denote the two blow-ups by
$$\beta_1,\beta_2\colon \bar{S} \ra \bP^2.$$
There are three conic bundle structures on $\bar{S}$, which we denote by
$$\gamma_1,\gamma_2,\gamma_3\colon \bar{S} \ra \bP^1.$$
They correspond
to the pairs of opposite sides on the hexagon; removing the opposing
sides, the remaining four sides form the degenerate fibers of the
associated conic bundle.
The symmetry group of the hexagon is $\fD_{12}\simeq \fS_2\times \fS_3$,
where the factors act on the $\{\beta_i\}$ and $\{\gamma_j\}$
respectively.

The Galois action on the Picard group gives a representation
$$\rho_S\colon \Gal(\bar{F}/F) \ra \fS_2 \times \fS_3.$$
Let $K/F$ denote the quadratic \'etale
algebra associated with the first factor and $L/F$ a cubic \'etale
algebra associated with the second factor.  Blunk \cite{Blunk}
(see also \cite[\S 4]{CTKM} and \cite{Corn})
has studied Azumaya algebras $B/K$ and $Q/L$ with the following properties:
\begin{itemize}
\item
The Brauer-Severi variety $\BS(B)$ has dimension two over $K$,
and there is a birational morphism
$$S_K \ra \BS(B)$$
realizing $S_K$ as the blow-up over an effective zero-cycle of degree three;
\item
$\BS(Q)$ has dimension one over $L$, and there is a morphism
$$S_L \ra \BS(Q)$$
realizing $S_L$ as a conic fibration with two degenerate fibers;
\item
the corestrictions $\cor_{K/F}(B)$ and $\cor_{L/F}(Q)$ are split
over $F$;
\item
$B$ and $Q$ both contain a copy of the compositum $KL$,
and thus are split over $KL$.
\end{itemize}

\noindent The following proposition is used in Section \ref{sect:rationality} to prove Theorem \ref{theo:main}.

\begin{prop} \label{prop:rational_section}
The following are equivalent:
\begin{enumerate}
\item $S$ is rational over $F$;
\item $S$ admits an $F$-rational point;
\item $S$ admits a zero-cycle of degree prime to six;
\item the Brauer classes $B$ and $Q$ are trivial.
\end{enumerate}
\end{prop}
\begin{proof}
The implications $(1) \Rightarrow (2) \Rightarrow (3)$ are clear.  The implication $(2) \Rightarrow (1)$ is \cite[Cor.~1 to Thm.~3.10]{ManinIHES}.  The implication $(3) \Rightarrow (4)$ is straightforward: note that $B$ is trivial if and only if $\BS(B)$ admits a zero-cycle of degree prime to 3, that $Q$ is trivial if and only if $\BS(Q)$ admits a zero-cycle of odd degree, and that $S_K$ maps to $\BS(B)$ and $S_L$ maps to $\BS(Q)$.  The implication $(4) \Rightarrow (2)$ is \cite[Cor.~3.5]{Blunk}.
\end{proof}

\begin{rema}
Blunk's classification \cite[Th.~3.4]{Blunk} implies
that $S$ may have maximal Galois
representation while $B$ and $Q$ are trivial. This may occur
over any field $F$ admitting a surjective representation
$$\rho\colon \Gal(\bar{F}/F) \ra \fS_2 \times \fS_3.$$
For instance, we could take $F=\bC(t)$ over which $B$ and $Q$
are necessarily trivial, as the Brauer group of 
a complex curve vanishes.
\end{rema}

\section{Rationality of our cubic fourfolds}
\label{sect:rationality}

\begin{proof}[Proof of Theorem~\ref{theo:main}]
Let $U \subset \cC_{18}$ be a Zariski open set parametrizing cubics fourfolds as in Theorem~\ref{theo:fibration}.  Given $X \in U$, we have the blow-up $r\colon X' \to X$ along $T$ and the sextic del Pezzo fibration $\pi\colon X' \to \bP^2$.  Let $S' \subset X'$ be a smooth fiber of $\pi$, and let $S = r(S')$.  We claim that if there is a cohomology class $\Sigma \in H^4(X,\bZ) \cap H^{2,2}(X)$ with $\Sigma.S = 1$, then $X$ is rational.  The integral Hodge conjecture holds for cubic fourfolds \cite[Thm.~1.4]{VoisinJAG}, so we can promote $\Sigma$ to an algebraic cycle.  Then $\Sigma' = r^* \Sigma$ satisfies $\Sigma'.S' = 1$ by the projection formula.  Passing to the generic fiber of $X'$, which is a smooth del Pezzo surface over the function field $\bC(\bP^2)$, we see that $\Sigma'$ defines a zero-cycle of degree 1, so Proposition~\ref{prop:rational_section} implies that $X'$ is rational over $\bC(\bP^2)$, hence over $\bC$, so $X$ is rational over $\bC$.

Thus our goal is to produce a countable dense set of divisors
$$\cC_K \subset \cC_{18}$$
such that if $X$ is in some $\cC_K$ then there is a Hodge class $\Sigma$ with $\Sigma.S = 1$.

Let $h \in H^2(X,\bZ)$ be the hyperplane class, let $L$ be the lattice underlying $H^4(X,\bZ)$, and let 
$L^0 = \langle h^2 \rangle^\perp \subset L$
be the primitive cohomology, which may be expressed \cite[Prop.~2.1.2]{Has00}:
\begin{equation} \label{eqn:lattice}
L^0 \simeq \left( \begin{matrix} 2 & 1 \\ 1 & 2 \end{matrix} \right)
\oplus \left( \begin{matrix} 0 & 1 \\  1 & 0  \end{matrix} \right)
\oplus \left( \begin{matrix} 0 & 1 \\  1 & 0  \end{matrix} \right)
\oplus E_8^{\oplus 2},
\end{equation}
where $E_8$ is the positive definite quadratic form associated with
the corresponding Dynkin diagram. Note that $L^0$ is even.

Consider positive definite rank-three extensions
\begin{equation} \label{eqn:rankthree}
\begin{array}{c|cc}
   & h^2 & S \\
\hline
h^2   &  3 & 6 \\
S   &  6 & 18 
\end{array} \subset 
K_{a,b} := \begin{array}{c|ccc}
       	& h^2 & S & \Sigma \\
\hline
h^2 & 3 & 6 & a \\
S  & 6 & 18 & 1 \\
\Sigma & a & 1 & b
\end{array}
\end{equation}
with discriminant 
$$\Delta=-3+12a-18a^2+18b.$$
The lattice $\langle h^2 \rangle^\perp \subset K_{a,b}$ is
even if and only if $a\equiv b \pmod{2}$.
We assume this parity condition from now on.

Replacing $\Sigma$ with $\Sigma+m(3h^2-S)$ for a suitable
$m \in \bZ$, we may assume that $a=-1,0,1$. Thus positive integers
$\Delta\equiv 9\pmod{12}$ arise as discriminants,
each for precisely one lattice, denoted $K_{\Delta}$ from
now on.

We claim that the embedding 
$$\langle h^2,S\rangle \hookrightarrow L$$
extends to an embedding of $K_{a,b}$ in $L$; compare \cite[\S4]{Has99}.
This is an application of Nikulin's results on embeddings of lattices \cite[\S1.14]{Nik}, 
which imply there is a unique vector
$v\in L^0$ satisfying
$$ v.v = 6, \ v.L^0 = \bZ,$$
modulo the automorphisms of $L^0$ acting trivially on the discriminant
group $d(L^0)=\Hom(L^0,\bZ)/L^0$. Thus we may assume
$$v:=S-2h^2=e_1+3f_1$$ 
where $\{e_j,f_j\},j=1,2,$ are the bases for the first and second
hyperbolic summands in (\ref{eqn:lattice}), respectively. 
For the case $a=0$ we may take $\Sigma=f_1+e_2+\frac{b}{2}f_2$,
as $b$ is necessarily even.
In the remaining cases, let $\{g_1,g_2\}$ be a basis for the first
summand in (\ref{eqn:lattice}) so that $g_1.g_1=g_2.g_2=2$, 
$g_1.g_2=1$, and 
$$w:=\frac{h^2+g_1+g_2}{3} \in L.$$
We are fixing the isomorphism $d(\langle h^2\rangle )\simeq d(L^0)$,
whose graph induces $L$ as an extension of $\langle h^2\rangle\oplus L^0$.
For $a=1$ we may take $\Sigma = w-f_1+e_2+\frac{b-1}{2}f_2$
and for $a=-1$ we choose $\Sigma = -w + e_1 + e_2+ \frac{b-1}{2}f_2$.
Here $b$ must be odd.

Excluding finitely many small $\Delta$, the lattice $K_{\Delta}$ defines a divisor
$$\cC_{K_{\Delta}}\subset \cC_{18}.$$
See \cite[\S 2.3]{Has16} for details on which $\Delta$ must be excluded, and \cite[\S4]{AT} for details of a similar calculation in $\cC_8$.

The divisors $\cC_{K_\Delta}$ intersecting $U$ are the ones we want.  The density in the Euclidean topology follows from the Torelli theorem \cite{Voisin86} and from \cite[5.3.4]{voisin-book}.
\end{proof}

Let $X$ be a cubic fourfold containing an elliptic ruled surface $T$ as in Theorem~\ref{theo:sexticcubic}, and let $\pi\colon X' \to \bP^2$ be the associated fibration into del Pezzo surfaces $S$ of degree $6$.  A {\em marking} of such a cubic
fourfold is a choice of lattice (see \cite[\S 3.1,5.2]{Has00})
$$\langle h^2,T\rangle =\langle h^2,S=4h^2-T\rangle  \subset A(X):= H^4(X,\bZ) 
\cap H^{2,2}(X)$$
associated with such a fibration. The marking data distinguishes
the classes $T$ and $S$ which are in the same orbit under the
monodromy action \cite[Prop.~5.2.1]{Has00}.

\begin{prop}[Hodge-theoretic interpretation] \label{prop:HTI}
Let $(X,T)$ be a marked cubic fourfold of discriminant $18$, and let
$\Lambda$ be the Hodge structure on the orthogonal complement of the 
marking lattice.
Then there exists an embedding of polarized Hodge structures
$$\Lambda(-1) \hookrightarrow H^2_\textrm{prim}(Y',\bZ),$$
where $(Y',f')$ is a polarized K3 surface of degree two
and the image of $\Lambda(-1)$ is an index-three sublattice.
This sublattice may be expressed in the form
\[
\langle\eta'\bmod 3\rangle^\perp,
\]
for some $\eta' \in H^2(Y',\bZ)$ whose image in $H^2(Y',\bZ/3\bZ)/\langle f'\rangle$ 
is isotropic under the intersection form modulo $3$.  

If $A(X)$ has rank $2$, then 
the class $\eta'$ gives rise to a non-trivial element in $\Br(Y')[3]$.
We may choose $(Y',f')$ and $\eta'$ so that 
$\eta' \equiv 0 \in \Br(Y')$ when the sextic del Pezzo
surface fibration
$\pi\colon X' \ra \bP^2$ admits a section $\Sigma$.
\end{prop}

\begin{proof}
The discriminant group of the lattice $\Lambda(-1)$ is isomorphic to $\bZ/3\bZ \oplus \bZ/6\bZ$ \cite[Prop.~3.2.5]{Has00} and Theorem~9 in \cite{MKSTVA} shows 
that there exist polarized K3 surfaces $(Y',f')$ of degree two such that $H^2_\textrm{prim}(Y',\bZ)$ contains a lattice $\Gamma$ of index three having the same rank and discriminant group as $\Lambda(-1)$.  The lattices $\Lambda(-1)$ and $\Gamma$ are even, indefinite and have few generators for their discriminant groups relative to their rank (two generators versus rank $21$). Results of Nikulin~\cite{Nik} now imply they are isometric.  The isomorphism
\begin{align*}
	H^2(Y',\bZ)/\langle f'\rangle \otimes \bZ/3\bZ &\xrightarrow{\sim} \Hom(\langle f'\rangle^\perp,\bZ/3\bZ) \\
	v\otimes 1 &\mapsto [t \mapsto v.t \bmod 3]
\end{align*}
identifies $\Gamma$ with the kernel of a map $\langle f'\rangle^\perp \to \bZ/3\bZ$ corresponding to some $\eta'\otimes 1$, so $\Gamma = \{t \in \langle f'\rangle^\perp : t.\eta' \in 3\bZ\}$; see~\cite[\S2.1]{vanGeemen} for details.  

If $A(X)$ has rank $2$, then $\NS(Y') = \langle f'\rangle$. The exponential sequence then gives rise to the following commutative diagram with exact rows:
\[
\xymatrix{
0 \ar[r] & H^2(Y',\bZ)/\langle f'\rangle \ar[r]\ar[d]^{\times 3} & H^2(Y',\cO_{Y'}) \ar[r]\ar[d]^{\times 3} & H^2(Y',\cO^*_{Y'}) \ar[r]\ar[d]^{\times 3} & 0 \\
0 \ar[r] & H^2(Y',\bZ)/\langle f'\rangle \ar[r] & H^2(Y',\cO_{Y'}) \ar[r] & H^2(Y',\cO^*_{Y'}) \ar[r] & 0
}
\]
The snake lemma then shows that to $[\eta'] \in H^2(Y',\bZ/3\bZ)/\langle f'\rangle$ there corresponds a class in $\Br(Y')[3]$.

For the last assertion, 
we compute discriminant groups in the notation developed above.
Write
$$d(\langle h^2, S\rangle) \simeq d(\Lambda) = \langle u_1,u_2\rangle$$
where
$$u_1 = \frac{1}{3}h^2 \simeq \frac{1}{3}(g_1+g_2)$$
and 
$$
u_2 = \frac{1}{6}(S-2h^2) = \frac{1}{6}(e_1+3f_1)\simeq \frac{1}{6}(3f_1-e_1).
$$
The $(\bQ/2\bZ)$-valued quadratic form on $d(\Lambda)$ is
$$u_1.u_1=\frac{2}{3}, \ u_1.u_2=0, \ u_2.u_2 = -\frac{1}{6}.$$
Nikulin's theory gives that even overlattices 
$$\Lambda' \supset \Lambda, \ |\Lambda'/\Lambda|=3$$
correspond to isotropic subgroups
$$ H \subset d(\Lambda), \ |H|=\bZ/3\bZ.$$
Note that $\Lambda'$ -- after a shift in weight and associated
sign reversal -- is the primitive cohomology of a degree two K3 surface.
There are two such subgroups, generated by $u_1+2u_2$ and $u_1-2u_2$;
we fix $H=\langle u_1-2u_2\rangle$. 

The class $\eta'\equiv 0$ in the Brauer group if it
can be represented by a divisor class.
This is the case provided
there is an algebraic class $\Xi \in L$ mapped, via the natural map
$$L \ra d(\langle h^2, S\rangle),$$ 
to a class $\bar{\Xi}$ such that 
$$\bar{\Xi} . H \neq  0 \in \bQ/\bZ.$$
Indeed, a suitable integral multiple of $\Xi$ projects
to a $\eta\in \Lambda'$ such that
$$\Lambda = \{t\in \Lambda': t.\eta \equiv 0\bmod{3} \}.$$
Using (\ref{eqn:rankthree}) we find that
$$\bar{\Sigma}=au_1 + (1-2a)u_2$$
and 
$$(au_1+(1-2a)u_2).(u_1-2u_2)=1/3.$$
\end{proof}

Most of 
Proposition~\ref{prop:HTI} could also have been obtained using the results on Mukai lattices in \cite[\S2]{huybrechts}.

\section{Singular fibers and discriminant curves}
\label{sect:fibrations}

In this section we study the locus where fibers of $\pi\colon X' \to \bP^2$ degenerate, and its relation to the Azumaya algebras $B/K$ and $Q/L$ of Section~\ref{sect:basics}.
\bigskip

We work over an algebraically closed field. Let $S$ be a {\em singular
del Pezzo surface} \cite[p.~29]{CT88}, essentially, a normal projective
surface with ADE singularities and ample anticanonical class.
Those of degree six ($K_S^2=6$) are classified as
follows \cite[Prop.~8.3]{CT88}:
\begin{itemize}
\item{type I: $S$ has one $A_1$ singularity and is obtained
by blowing up $\bP^2$ in three collinear points, and blowing down
the proper transform of the line containing them;}
\item{type II: $S$ also has one $A_1$ singularity and is obtained by
blowing up $\bP^2$ in two infinitely near points and a third point
not on the line associated with the infinitely near points,
then blowing down the proper transform of the first exceptional
divisor over the infinitely near points;}
\item{type III: $S$ has two $A_1$ singularities and is obtained by blowing
up two infinitely near points and a third point all contained in a line,
and blowing down the proper transforms of the first exceptional divisor
over the infinitely near point and the line;}
\item{type IV: $S$ has an $A_2$ singularity and is obtained by
blowing up a curvilinear subscheme of length three not contained in a 
line and blowing down the first two exceptional divisors;}
\item{type V: $S$ has an $A_1$ and and $A_2$ singularity and is obtained
by blowing up a curvilinear subscheme contained in a line, and blowing
down the proper transforms of the first two exceptional divisors and the
line.}
\end{itemize}
The classification proceeds by analyzing the $(-2)$-classes 
contracted in each case and the subgroups of 
the Weyl group $\fS_2\times \fS_3$
generated by the corresponding roots.
Types I and II correspond to conjugacy 
classes of transpositions in the $\fS_2$ and $\fS_3$
factors.
Type III corresponds to conjugacy
classes of Klein four-groups.
Type IV corresponds to the $\fS_3$-factor and 
type V to the
the full group. 
Each (possibly) singular sextic del Pezzo surface anti-canonically
embeds in $\bP^6$ \cite[Prop.~0.6]{CT88}.

Let $\cDP$ denote the moduli stack of (possibly)
singular sextic del Pezzo surfaces. 
It is a non-separated 
Artin stack of dimension $-2$, as all the surfaces
have automorphism groups of dimension at least two.
The stack $\cDP$ may be constructed
as a quotient of the
Hilbert scheme parametrizing sextic surfaces in $\bP^6$ anti-canonically
embedded with
the requisite singularities.
Another atlas is given by taking anti-canonical images of iterated blow-ups
of $\bP^2$ along three points, as described above.
Both presentations show that types I and II occur in codimension one, types
III and IV in codimension two, and type V in codimension three.

\begin{defi} \label{defi:good}
Let $P$ be a smooth complex projective surface. 
A {\em good sextic del Pezzo fibration} consists of a smooth
fourfold $\cS$ and a flat projective morphism $\pi\colon \cS \ra P$
such that the
fibers of $\pi$ are either smooth or singular sextic del Pezzo surfaces.
Let $B_I, \dotsc, B_V$ denote the closure of the corresponding loci in $P$;
we require that:
\begin{itemize}
\item $B_I$ is a non-singular curve.
\item $B_{II}$ is a curve, non-singular away from $B_{IV}$.
\item $B_{III}$ is finite and coincides with the intersection of $B_I$ and $B_{II}$, which is transverse.
\item $B_{IV}$ is finite, and $B_{II}$ has cusps at $B_{IV}$.
\item $B_{V}$ is empty.
\end{itemize}
We allow $B_I,B_{II},B_{III},$ or $B_{IV}$ to be empty.
\end{defi}
We call $B_I\cup B_{II}$ the {\em discriminant divisor} of $\pi$;
our usage is consistent with the standard notion of the discriminant
of a family of varieties with isolated hypersurface singularities.
Given such a family with smooth total space, the multiplicity
of the discriminant at a point equals the sum of the Milnor numbers of the
singularities of the corresponding fiber \cite[2.8.3]{Teissier}.
Thus the conditions on the singularities of the discriminant force
the singularities of the fibers to be mild.

Here is some motivation for Definition~\ref{defi:good}:
\begin{prop} \label{prop:versality}
Assume the characteristic is not $2$ or $3$.
Under the conditions of Definition~\ref{defi:good},
the classifying map $P\rightarrow \cDP$ is transverse to each
equisingular stratum,
i.e., the localization of $P$ along each stratum is versal for the
corresponding singularities.
Thus good sextic del Pezzo fibrations are Zariski
open in the moduli space of all sextic del Pezzo fibrations with fixed
invariants.
\end{prop}
This is an application of the following:
\begin{lemm}
Assume the characteristic is not $2$ or $3$.
Let $P$ be a smooth surface and $\pi\colon \cS \ra P$ a flat morphism
of relative dimension two such that
the fibers of $\pi$ are smooth, have one or two $A_1$ singularities,
or have one $A_2$ singularity. Then the following are equivalent:
\begin{enumerate}
\item{$\pi$ is versal for each equisingular singular stratum;}
\item{$\pi$ satisfies the following:
\begin{itemize}
\item{the discriminant of $\pi$ in $P$ is reduced;}
\item{the discriminant has nodes at points with two $A_1$ singularities
and cusps at points with $A_2$ singularities.}
\end{itemize}}
\end{enumerate}
The condition that the discriminant of $\pi$ is reduced
follows from the smoothness of the fourfold $\cS$.
Either of the enumerated conditions remains valid for 
small deformations of $(P,\cS,\pi)$.
\end{lemm}
\begin{proof}
The last statement is `openness of versality' \cite[Th.~4.4]{Art74};
it suffices then to focus on the formal deformations.

The versal deformation of an $A_1$
surface singularity takes the form
$$xy+z^2 = t,$$
where the base of the deformation is the spectrum of a complete
discrete valuation ring with uniformizer $t$, which is the discriminant.
A general one-parameter deformation with parameter $\tau$
is formally equivalent to
\begin{equation} \label{eqn:general}
xy+z^2 = f(\tau), f(0)=0,
\end{equation}
with $t=f(\tau)$ the classifying map. 
The following are equivalent:
\begin{itemize}
\item{$f$ is unramified;}
\item{the deformation (\ref{eqn:general}) is versal;}
\item{the discriminant of (\ref{eqn:general}) is reduced;}
\item{the total space of (\ref{eqn:general}) is smooth.}
\end{itemize}
When there are two $A_1$ singularities, the versality condition is that
the corresponding branches of the discriminant are transverse,
whence the nodes.

We focus on the $A_2$ case; we refer the reader to \cite{GK}
for the classification in arbitrary characteristic
and \cite[p.~434]{DH88} for the explicit
computation of the discriminant of a versal deformation.
Here we use the assumption on the characteristic. It remains to check
that a deformation of an $A_2$ singularity with (reduced) cuspidal 
discriminant is necessarily versal. Let $0\in P$ be the corresponding
point in the base. If versality fails then
the classifying
map to the versal deformation ramifies at $0$ but not at the 
generic point of the discriminant. This means there exists
a smooth curve $0 \in C \subset S$ such that the classifying map 
takes $C$ birationally onto its image but ramifies at 0.
Since the discriminant is cuspidal, $C$ may be chosen to intersect
the discriminant with multiplicity two or three. 
The image of $C$ in the versal deformation is singular and
thus meets the discriminant in degree at least four, a contradiction.
\end{proof}

The singularity analysis guarantees that, for good sextic del Pezzo fibrations, the local monodromies at each stratum are
maximal. At generic points of $B_I$ and $B_{II}$ we have order-two monodromy,
at points of $B_{III}$ we get $\fS_2 \times \fS_2$, and at points of
$B_{IV}$ we see $\fS_3$. 

Recall the notation introduced in Section~\ref{sect:basics}:
\begin{prop} \label{prop:blunk}
Let $\pi\colon \cS\ra P$ be a good sextic del Pezzo fibration
over $\bC$. Then Blunk's construction yields:
\begin{itemize}
\item{a non-singular double cover $Y \ra P$ branched along $B_I$,
representing $K$;}
\item{a subgroup $\langle\eta\rangle \subset \Br(Y)[3]$, representing $B$;}
\item{a non-singular degree-three cover $Z \ra P$ branched along $B_{II}$,
representing $L$;}
\item{an element $\zeta \in \Br(Z)[2]$, representing $Q$.}
\end{itemize}
\end{prop}
\begin{proof}
Let $Y$ and $Z$
denote the normalizations of $P$ in
$K$ and $L$ respectively.
Our strategy is to relate $Y$ and $Z$ to geometric constructions
associated with
$\pi\colon \cS \ra P$ which carry
the structure of Brauer-Severi schemes over these surfaces. 

We first address the double cover. Recall the
interpretations of $K$:
\begin{itemize}
\item{the fixed algebra of $\fS_3 \subset \fS_2 \times \fS_3$ as it acts on $KL$;}
\item{the algebra of definition of the blow-up realizations
of the geometric generic fiber
$\beta_1,\beta_2\colon \bar{S} \ra \bP^2;$}
\item{an extension trivializing the monodromy about
singular surfaces of type I.}
\end{itemize}
Thus $Y\ra P$ is simply branched along $B_I$ and hence is smooth.

Consider the pull-back family
$$\cS\times_P Y \ra Y.$$
It is singular along a curve $C_I$, the closure of the $A_1$ singularities
associated with $B_I$; a local computation shows these are
ordinary threefold double points at the generic point of $C_I$.
Consider the blow-up 
$$\cS_I:=\Bl_{C_I}(\cS\times_P Y) \ra Y;$$
its fiber over a general point $c\in C_I$ is the union 
$$\Bl_c(\cS_b) \cup_{L_b} \bF_0,$$
where $b \in B_I$ is the image of $c$, $\cS_b=\pi^{-1}(b)$, and
$\bF_0 \simeq \bP^1 \times \bP^1$.
Since $\cS_b$ is type I, $\Bl_c(\cS_b)$ is isomorphic to $\bP^2$ blown
up at three collinear points and $L_b$ is the proper transform of the
line containing them. 
The class of $L_b \subset \bF_0\simeq \bP^1 \times \bP^1$ is
equivalent to the diagonal.

\begin{lemm}
There exists an open subset $U\subset Y$ with finite complement
over which $\cS_I$ is isomorphic to the blow-up of a
Brauer-Severi fibration over $U$ restricting to $\BS(B)$
over the generic point.
\end{lemm}
\begin{proof}[Proof (of Lemma)]
Choose $U$ to exclude the preimages of $B_{III}$ and $B_{IV}$
and any other points of $B_I$ where the description above 
fails to hold.

Over the generic point of $B_{II}$,
the contractions $\beta_1$ and $\beta_2$ are well-defined
morphisms, induced by Cartier divisors on the underlying
type II surfaces. Indeed, a type II surface is obtained by
blowing up $\bP^2$ at three points, two of which are infinitely 
near, and blowing down the proper transform of the line through these;
the two blow-up realizations are visible.

Type I surfaces $\cS_b$ only admit one rational
map $\beta\colon \cS_b \dashrightarrow \bP^2$, with indeterminacy
at the $A_1$ singularity. But the blow-up described above resolves
this. The desired blowdown acts on 
$\Bl_c(\cS_b) \cup_{L_b} \bF_0$ by blowing down $\Bl_c(\cS_b)$
to $\bP^2$ and collapsing $\bF_0$ to one of its $\bP^1$ factors. 

Let $L_1$ and $L_2$ denote Cartier divisors on $\cS_I \times_P U$
inducing $\beta_1$ and $\beta_2$, which make sense even over the
generic point of $B_I$ by the previous paragraph.
Fiberwise computation \cite{Har} shows that the divisors $L_1$ and $L_2$ have
vanishing higher direct images and are basepoint free over $U$.
We construct the Brauer-Severi fibration by taking 
relative $\Proj$ of the graded $\cO_U$-algebra given by the pushforward of $\bigoplus_m \cO_{\cS_I}(m L_1)$, or we could equally use $L_2$.
Over the generic point, this blows down three disjoint $(-1)$-curves; over
$B_I$ we also collapse the $\bF_0$ component as indicated above.  
These Brauer-Severi fibrations give rise to inverse elements in
$\Br(U)[3]$ by~\cite[Theorem~2.3]{Corn}.
\end{proof}
The classes in $\eta \in \Br(U)[3]$ given by the lemma extend
to $Y$ via purity \cite[Th.~6.1]{Gro}.

\

We turn to the triple cover. Recall the
interpretations of $L$:
\begin{itemize}
\item{the fixed algebra of 
$$\fS_2 \times \langle \tau \rangle
\subset \fS_2 \times \fS_3,$$
as it acts on $KL$, where $\tau$ is a transposition;}
\item{the algebra of definition for one of the conic bundle realizations
of the geometric generic fiber
$\gamma_1,\gamma_2,\gamma_3\colon \bar{S} \ra \bP^1.$}
\end{itemize}
Moreover, the transposition $\tau$ corresponds to a root of
$\fS_2\times \fS_3$ associated with type II singular fibers.
This identification shows that $Z$ is simply branched over $B_{II}$.
The interpretation of type IV fibers via subgroups of $\fS_2 \times \fS_3$
shows that $Z$ is totally branched along $B_{IV}$. 
The standard theory of triple covers \cite[\S 5]{Mir} implies that
$Z$ is non-singular. 

Consider the pull-back family
$$\cS\times_P Z \ra Z.$$
It is singular along a curve $C_{II}$, the closure of the $A_1$ singularities
associated with $B_{II}$, with 
ordinary threefold double points at the generic point of $C_I$.
The fibers over the generic point $c\in C_{II}$ are type II surfaces $\cS_b$,
where $b \in B_{II}$ is the image of $c$.

Number the conic bundles so that $\gamma_1$ and $\gamma_2$ coincide at $c \in C_{II}$, 
i.e., $L$ is the algebra of definition for $\gamma_3$. Our geometric description of
type II surfaces implies
\begin{itemize}
\item{$\gamma_1=\gamma_2\colon \cS_b \dashrightarrow \bP^1$, a rational map
with indeterminacy at the $A_1$ singularity;}
\item{$\gamma_3\colon\cS_b \ra \bP^1$ is well defined.}
\end{itemize}
The description of type I surfaces implies that the
three conic bundles
$$\gamma_1,\gamma_2,\gamma_3:\cS_b \dashrightarrow \bP^1$$
are distinct and define morphisms over the smooth locus of $\cS_b$,
with the singularity as base point.  
Where the conic bundle is a rational map, the general fibers
are smooth curves on $\cS_b$ passing through an $A_1$ singularity,
hence $2$-Cartier but not Cartier.

As before, we choose $V \subset Z$ with codimension-two complement,
excluding points mapping to $B_{III}$ or $B_{IV}$ but meeting
$B_I$ and $B_{II}$. We construct the conic bundle over $V$ as the
base $\bP^1$s parametrizing the families of conics.
This may be interpreted through linear series: There exists
a Cartier divisor $M$ on $\cS\times_P V$ -- twice the class of the corresponding
conic fibers. It has the following properties:
\begin{itemize} 
\item{for $v\in V$ not lying over $B_I$ or in the ramification locus 
of $V \ra P$, $M$ induces the tautological
conic bundle $\cS_v \ra \bP^1$ with the base realized as a conic curve;}
\item{for $v\in V$ lying over $B_I$ or in the ramification locus, it induces the tautological
rational conic bundle $\cS_v \dashrightarrow \bP^1$, again with the
base realized as a smooth conic curve.}
\end{itemize}
We construct the Brauer-Severi fibration by taking the
$\Proj$ for $M$ on $\cS\times_P V$ relative to $V$, similar to what we did with $L_1$ or $L_2$ over $U$.

Let $\zeta \in \Br(V)[2]$
be the corresponding class, which extends by purity to an
element of $\Br(Z)$.
\end{proof}

\begin{prop} \label{prop:Eulergeneral}
Let $\pi\colon\cS \ra P$ be a good del Pezzo fibration and fix 
$$\begin{array}{rcl}
b_{IV} & = & \chi(B_{IV})=|B_{IV}| \\
b_{III} & = & \chi(B_{III})=|B_{III}|=|B_{I} \cap B_{II}| \\
b_{II} & = & \chi(B_{II}\setminus (B_{III}\cup B_{IV})) \\
b_{I} & = & \chi(B_{I}\setminus B_{III}),
\end{array}
$$
where $\chi$ is the topological Euler characteristic.
Then we have
$$\chi(\cS)=6\chi(P)-b_I-b_{II}-2b_{III}-2b_{IV}.$$
\end{prop}
\begin{proof}
This follows from the stratification of the fibration by singularity type in Definition~\ref{defi:good}.
A smooth sextic del Pezzo surface has $\chi=6$. 
For types I and II we have $\chi=5$; for types III and IV we have
$\chi=4$.
\end{proof} 

We specialize Proposition~\ref{prop:Eulergeneral} to the case where
the base is $\bP^2$, using Bezout's Theorem and the genus formula:
\begin{coro} \label{coro:Euler}
Let $\pi\colon\cS \ra \bP^2$ be a good del Pezzo fibration;
write $d_I=\deg(B_I)$ and $d_{II}=\deg(B_{II})$. Then we
have
$$\chi(\cS)=14+(d_I-1)(d_I-2)+(d_{II}-1)(d_{II}-2)-3b_{IV}.$$
\end{coro}

\

We specialize further to the del Pezzo fibrations appearing in our main construction:

\begin{prop}
Let $\pi\colon X' \to \bP^2$ be as in Theorem~\ref{theo:fibration}.  The discriminant locus contains a sextic curve with nine cusps, projectively dual to a smooth plane cubic.
\end{prop}
\begin{proof}
Again let
\[ \xymatrix{
\tilde\bP^5 \ar[d]_-q \ar[r]^-p & \Pi_1 \times \Pi_2 \\
\bP^5
} \]
be as in the proof of Proposition~\ref{prop:two_planes}, and recall that $p$ is a $\bP^1$-bundle and that $E \subset \Pi_1 \times \Pi_2$ was obtained as the intersection of three $(1,1)$-divisors.  Let $\Pi_3 = |E| \cong \bP^2$ be the linear system that they span.  If we view $\Pi_1 \times \Pi_2$ as embedded in $\bP^8$ via the Segre embedding, then $\Pi_3$ is embedded in the dual $\bP^8$.  By projective duality, the singular $(1,1)$-divisors are parametrized by a smooth determinantal cubic curve $E_3 \subset \Pi_3$.  The linear system of quadrics containing $T$ is identified with $\Pi_3$, so the map $\pi$ takes values in $\Pi_3^\vee$, and we claim that over the dual curve $E_3^\vee \subset \Pi_3^\vee$, the fibers of $\pi$ are singular.

Indeed, a line $L \subset \Pi_3$ determines a complete intersection of two $(1,1)$-divisors $W_L \subset \Pi_1 \times \Pi_2$. By standard arguments, $W_L$ is singular if and only if $L$ is tangent to $E_3$, and the singularities of $W_L$ are disjoint from $E$.  Thus $\tilde V_L = p^{-1}(W_L)$ is singular along a line disjoint from $\tilde T = p^{-1}(E)$, and $V_L = q(\tilde V_L)$ is singular along a line that meets $T = q(\tilde T)$ in only two points: one contained in $E_1 = T \cap \Pi_1$, and one in $E_2$.  The cubic $X$ meets this line in these two points and at least one more, yielding a singularity of the residual surface $S_L$, which is the fiber of $\pi$ over the point $L^\vee \in \Pi_3^\vee$.

The nine cusps of $E_3^\vee$ correspond to the nine inflection points of the cubic $E_3$.
\end{proof}

\begin{remas}
(a) The net of $(1,1)$-divisors containing $E$ is identified with the net of quadrics containing $T$, and the discriminant sextic of the latter is the square of the discriminant cubic of the former.

(b) The relation between the elliptic curves $E \subset \Pi_1 \times \Pi_2$ and $E_1, E_2, E_3$ in $\Pi_1, \Pi_2, \Pi_3$ is explained in \cite{ng} or \cite[\S5.2]{bhargava-ho}.

(c) The two trisections of $\pi\colon X' \to \bP^2$ are also ramified over this cuspidal sextic, hence are likely isomorphic to the triple cover $Z$ of Proposition~\ref{prop:blunk}.  We also expect that the elliptic ruled surface $T$ is isomorphic to $Z$.  Clearly there is a lot of interesting geometry to explore here.
\end{remas}

\begin{prop}
If the fibration $\pi\colon X' \to \bP^2$ is good in the sense of Definition~\ref{defi:good}, then the curve $B_I \subset \bP^2$ is a sextic, and the double cover $Y$ is a K3 surface of degree 2.
\end{prop}
\begin{proof}
Since $X' = \Bl_T(X)$ and $\chi(T) = 0$, we have
\[ \chi(X') = \chi(X) = 1 + 1 + 23 + 1 + 1 = 27. \]
In the previous proposition we have seen that $d_{II} = 6$ and $b_{IV} = 9$.  Thus from Corollary~\ref{coro:Euler} we find that $d_I = 6$.  For the last statement, recall that $Y$ is the double cover of $\bP^2$ branched over $B_I$.
\end{proof}

\begin{rema} Let $f$ denote the degree-two polarization on $Y$ associated with
the double cover. We expect that $(Y,f,\eta)$ coincides (up to sign) with the
triple $(Y',f',\eta')$ obtained (via Hodge theory) at the end of
Proposition~\ref{prop:HTI}.
\end{rema}

In the next section we will see that there are cubics $X \in \cC_{18}$ for which $\pi\colon X' \to \bP^2$ is good.

\section{An explicit example}
\label{sect:example}

The computations below were verified symbolically with Macaulay2 \cite{M2}.  Code is available on the arXiv as an ancillary file.  For speed of computation we work over the finite field $\bF_{5}$, explaining at the end of the section how to lift to characteristic zero.

Let $\bP^5 = \Proj(\bF_{5}[x_0,\dots,x_5])$.  Define quadrics
	\begin{align*}
		Q_1	&= 3x_0x_3 + 2x_0x_5 + 4x_1x_3 + 2x_1x_4 + x_2x_3 + x_2x_4 + 2x_2x_5;\\
		Q_2	&= x_0x_3 + 2x_0x_5 + x_1x_3 + 3x_1x_5 + 2x_2x_4 + 3x_2x_5;\\
		Q_3	&= 2x_0x_4 + x_0x_5 + x_1x_3 + 2x_1x_5 + 4x_2x_3 + 3x_2x_5.
	\end{align*}
	Each quadric contains the planes
	\[
	\{x_0 = x_1 = x_2 = 0\}\qquad\text{and}\qquad\{x_3 = x_4 = x_5 = 0\}.
	\]
	The sextic elliptic ruled surface $T$, obtained by saturating the ideal generated by the three quadrics with respect to the defining ideals of the planes, is cut out by $Q_1, Q_2, Q_3$, and the two cubics
	\begin{align*}
      x_3^3 + 2x_3x_4^2 + x_3x_4x_5 + 4x_3x_5^2 + 4x_4^3 + 4x_5^3  &= 0,\\
      x_0^3 + 4x_0^2x_1 + x_0^2x_2 + 2x_0x_1^2 + 2x_0x_1x_2 + 4x_0x_2^2 + x_1^3 + 3x_1^2x_2 + x_2^3 &= 0.
	\end{align*}
	
	The surface $T$ is contained in the cubic fourfold $X$ cut out by
	\[
	\begin{split}
		f	&:= x_0^3 + 4x_0^2x_1 + x_0^2x_2 + x_0^2x_3 + 3x_0^2x_4 + 3x_0^2x_5 + 2x_0x_1^2 + 2x_0x_1x_2 \\
			&\qquad + 4x_0x_1x_3 + 3x_0x_1x_4 + 4x_0x_1x_5 + 4x_0x_2^2 + x_0x_2x_3 + 3x_0x_2x_4 \\
			&\qquad + 2x_0x_2x_5 + 3x_0x_3^2 + 4x_0x_3x_5 + 4x_0x_4^2 + 2x_0x_4x_5 + x_0x_5^2 + x_1^3 \\
			&\qquad + 3x_1^2x_2 + 4x_1^2x_3 + x_1x_2x_3 + 3x_1x_2x_4 + 4x_1x_2x_5 + 3x_1x_3^2 \\
			&\qquad + x_1x_3x_4 + 2x_1x_4^2 + x_1x_4x_5 + 2x_1x_5^2 + x_2^3 + 4x_2^2x_3 + x_2^2x_4 \\
			&\qquad + 4x_2^2x_5 + 4x_2x_3^2 + 3x_2x_3x_5 + 3x_2x_4^2 + 2x_2x_4x_5 + 4x_2x_5^2 + 4x_3^3 \\
			&\qquad + 3x_3x_4^2 + 4x_3x_4x_5 + x_3x_5^2 + x_4^3 + x_5^3	
    	\end{split}
	\]
	A direct computation of the partial derivatives of $f$ shows that $X$ is smooth.  The first order deformations of $T$ as a subscheme of $X$ are given by
\[
\Gamma(T,\cN_{T/X})=\Hom(\cI_T,\cO_T).
\]
Macaulay2 verifies that this is one-dimensional, as was required at the end of the proof of Theorem~\ref{theo:sexticcubic}.

	The discriminant locus of the map $\pi\colon X' := \Bl_T(X) \to \bP^2 = \left|\cI_T(2)\right|^\vee$ is a reducible curve of degree $12$, with two irreducible components:
	\begin{align*}
		B_I		&: x^6 + 2x^4y^2 + x^3y^3 + 4x^3y^2z + 2x^3z^3 + 4x^2y^4 + 4x^2y^2z^2 \\
				&\qquad + 4x^2yz^3 + 4xy^5 + xy^4z + xy^2z^3 + xyz^4 + 2xz^5 + 4y^6 \\
				&\qquad + 3y^5z + y^3z^3 + y^2z^4 + 4yz^5 = 0, \\
      		B_{II}	&:  x^6 + 2x^5y + 2x^4y^2 + x^4yz + 4x^3y^3 + 3x^3y^2z + 4x^3yz^2 + x^3z^3 \\
				&\qquad + 3x^2y^4 + 4x^2y^2z^2 + x^2yz^3 + 3x^2z^4 + 3xy^5 + 2xy^4z \\
				&\qquad + 3xy^3z^2 + 3xyz^4 + xz^5 + y^5z + 4y^4z^2 + 3y^3z^3 \\
				&\qquad + 2y^2z^4 + 4yz^5 = 0
	\end{align*}
	The curve $B_I$ is smooth, and $B_{II}$ has $9$ cusps.  Their intersection is a reduced $0$-dimensional scheme of degree $36$ and is thus transverse.  Thus $\pi$ is good in the sense of Definition~\ref{defi:good}.

Since the relevant Hilbert schemes are smooth and rational, the equations
we write down readily lift to characteristic zero. The properties we
stipulate are open (see Proposition~\ref{prop:versality}) and thus hold for any such lift.

\begin{note}[Added in revision]
We point out recent results since
this paper was first written:
Kontsevich and Tschinkel \cite{KT} have shown that rationality specializes in 
families of smooth projective varieties; thus the rationality assertion of
Theorem~\ref{theo:main} holds on $\cup \cC_K$, not just over an open subset.
Kuznetsov \cite{Kuz17} has further developed the theory of sextic del Pezzo surfaces
and their degenerations, with a view toward applications to derived categories.
Russo and Staglian\`o \cite{RS17,RS18} have proposed new constructions for rational cubic fourfolds
applicable over two divisors in the moduli space.

\end{note}

\bibliographystyle{alpha}
\bibliography{delPezzoExample}

\end{document}